\documentclass[11pt,twoside]{amsart}

\usepackage[english]{babel}
\usepackage{amsmath, amsthm, amssymb, amsfonts}
\usepackage{latexsym}
\usepackage{graphicx}
\usepackage{pdfpages}
\usepackage {bm}
\usepackage {indentfirst} 

\usepackage{hyperref,cite}

\advance\hoffset by -1.5cm

\newcommand{\T}{\mathbb{T}}
\newcommand{\D}{\mathbb{D}}
\newcommand{\C}{\mathbb{C}}
\newcommand{\N}{\mathbb{N}}

\newcommand{\Z}{\mathbb{Z}}

\newcommand{\B}{\mathcal{B(H)}}
\newcommand{\R}{\mathcal{R}}
\newcommand{\n}{\mathcal{N}}
\newcommand{\h}{\mathcal{H}}
\newcommand{\e}{\mathcal{E}}

\newcommand{\ka}{\mathcal{K}}

\renewcommand{\Re}{\operatorname{Re}}

\newcommand{\Dom}{\mathop{\rm Dom}}

\newtheorem{theorem}{Theorem}[section]
\newtheorem{lemma}[theorem]{Lemma}
\newtheorem{proposition}[theorem]{Proposition}
\newtheorem{corollary}[theorem]{Corollary}

\theoremstyle{definition}

\newtheorem{example}[theorem]{Example}

\theoremstyle{definition}
\newtheorem{remark}[theorem]{Remark}

\newcommand{\Prec}{\stackrel{\rm H}{\prec}}
\newcommand{\PrecZ}{\stackrel{\rm Z}{\prec}}

\numberwithin{equation}{section}

\title{Classes of contractions and Harnack domination}

\author[C. Badea]{Catalin Badea}
\address{ UFR de Math\'ematiques, Laboratoire Paul Painlev\'e, UMR CNRS 8524,
	Universit\'e Lille 1, 59655 Villeneuve d'Ascq Cedex, France}
\email{badea@math.univ-lille1.fr}

\author[L. Suciu]{Laurian Suciu}
\address{Department of Mathematics, "Lucian Blaga" University
	of Sibiu, Dr. Ion Ra\c tiu 5-7, Sibiu, 550012, Romania}
\email{laurians2002@yahoo.com}

\author[D. Timotin]{Dan Timotin}
\address{Institute of Mathematics Simion Stoilow of the Romanian Academy, P.O.
	Box 1-764, 014700 Bucharest, Romania}
\email{Dan.Timotin@imar.ro}



\begin{document}

\begin{abstract}
Several properties of the Harnack domination of linear operators acting on Hilbert space with
		norm less or equal than one are studied. Thus, the maximal elements for this relation
		are identified as precisely the singular unitary operators, while the minimal elements are shown to be
		the isometries and the adjoints of isometries.
		We also show how a large range of properties (e.g. convergence of iterates, peripheral spectrum, ergodic properties)
		are transfered from a contraction to
		one that  Harnack dominates it.
\end{abstract}
\maketitle

\section{Introduction}  \label{Section:Intro}
\medskip

The classical Harnack inequality
for positive harmonic functions in the unit disc was generalized to some operator inequalities
for contractions (linear operators of norm no greater than one) on Hilbert space by Ion~Suciu in the 1970s. Using this
generalized inequality, a preorder relation for Hilbert space contractions, called the Harnack domination, has been introduced in
\cite{S1,S2}.
Notice also that different operator theoretical generalizations of the Harnack inequality have been proved by
Ky Fan (see \cite{kF} and the references therein); we will not consider these generalizations here.

The Harnack preorder condition between two contractions can be expressed in several equivalent forms:
majorization of the associated operator Poisson kernels, certain positive-definiteness
conditions or majorization of the semi-spectral measures (cf. Theorem \ref{thm:1.1} below). It has both analytic and geometric consequences. The preorder given by
Harnack domination
induces an equivalence relation, the corresponding equivalence classes being the Harnack parts.
The concept of Harnack parts, as well as the hyperbolic metric defined in \cite{SV}, are the analogues in the noncommutative
case of the Gleason parts and metric defined in the context of function algebras.
Different aspects of the Harnack domination of contractions have been studied by
several authors \cite{AST,C,F,KSS,S1,S2,S3,SS,Sn,SV}. An extension of Harnack domination to the operators of class $C_{\rho}$ (that is $\rho$-contractions) in the sense of \cite{SFb} appears in \cite{CS}, while in \cite{Po1, Po2} the Harnack domination in the non-commutative unit ball, or $C_{\rho}$-ball of $\B^n$ for $n>1$ was studied.

The aim of the present paper is to study several properties of Harnack domination of
contractions on a Hilbert space.  We identify the maximal elements for this relation
as precisely the singular unitary operators. We prove that the minimal elements are the isometries and the
coisometries (adjoints of isometries). We also show how a
large range of properties are transferred from a contraction to one
that Harnack dominates it.
A useful tool is the asymptotic limit $S_T$, defined as
the strong limit of the sequence
 $\{T^{*n}T^n\}_{n\in \N}$.

 The plan of the paper is the following. Section~2 is devoted
to different preliminary definitions and results. Among
other we include a new characterization of Harnack domination of an
isometry by a contraction, which is useful in the sequel. This characterization is in terms of
the behaviour of the resolvent of one operator applied to the difference of the two
operators and quickly gives the characterization of
minimal elements for the Harnack domination. In Section~3
we find the maximal elements, while Section~4 investigates the effect of
Harnack domination on certain ergodic properties as well as on the peripheral
spectrum of a contraction. In Section~5 we show how different
classes of operators are preserved by Harnack domination. The final section contains several
examples, one of them showing some spectral and structural properties
which are not preserved by Harnack domination.

\medskip

{\bf Acknowledgments.} The first author has been partially supported by the ``Laboratoire d'excellence''
Labex CEMPI  (ANR-11-LABX-0007-01) and by the EU IRSES grant PIRSES-GA-2012-318910, AOS.
The first and second author were partially supported
 by the ``Laboratoire Europ\'een Associ\'e CNRS Franco-Roumain'' MathMode.
The third author was partially supported by a grant of the Romanian National Authority for Scientific
Research, CNCS – UEFISCDI, project number PN-II-ID-PCE-2011-3-0119.

\section{Notations and preliminaries} \label{Section:Notation}

In the sequel $T,T' \in \B$ will be linear contractions acting on the complex Hilbert space $\h$;
$V$ acting on $\ka$ and $V'$ acting on $\ka'$ will denote
the minimal isometric dilations of $T$ and $T'$ respectively.
$\n(T)$ and $\R(T)$ stand for the kernel and respectively
the range of the operator $T$.
 We shall denote by $I$ the identity operator on $\h$ and by
$$T_{\lambda}  = (T-\lambda I)(I - \overline{\lambda}T)^{-1}$$
the M\"obius
transform of $T$. Here $\lambda$ is an element
of the open unit disk $\D$. For a contraction $T$ we denote by $D_T = (I-T^{\ast}T)^{1/2}$ the \emph{defect operator} and by $\mathcal{D}_T=\overline{\R(D_T)}$ the \emph{defect space} of $T$.  The \emph{Poisson kernel} of $T$ is
$$ K(T,\lambda) = (I - \overline{\lambda}T)^{-1} + (I - \lambda T^{\ast})^{-1} -I .$$
As
$$ K(T,\lambda) = (I - \lambda T^{\ast})^{-1}(I-|\lambda|^2T^{\ast}T)(I - \overline{\lambda}T)^{-1}$$
and $\|T\| \le 1$,
the Poisson kernel is a positive operator in the sense that
$$\langle K(T,\lambda)h,h\rangle \ge 0\quad (h\in \h, \lambda \in \D) .$$
We also consider the operators
$$ T^{[k]} =
\left\{
\begin{array} {r@{\quad:\quad}l} T^{k} & k \ge 0 \\
T^{\ast |k|} & k < 0 \end{array} \right. .$$

 The \emph{asymptotic limit} $S_T \in \B$ of the
 contraction $T$ (see, for instance,~\cite[Chapter 3]{K}) is the strong limit of the sequence
 $\{T^{*n}T^n\}_{n\in \N}$.  It is a positive contraction with $\|S_T\|=1$ whenever $S_T\neq 0$.
Notice that $\n(I-S_T)=\bigcap_{n\ge1} \n(I-T^{*n}T^n)$ is the
 maximal invariant subspace (of $\h$) for $T$ on which $T$ is an
 isometry, while $\n(I-S_T)\cap \n(I-S_{T^*})$ is the maximal reducing subspace for $T$ on which $T$ is unitary.

 We say that $T$ is \emph{strongly (weakly) stable} if the sequence
 $\{T^n\}_{n\in \N}$ is strongly (weakly) convergent to $0$ in $\B$ (see, for instance,~\cite{K}).
 Also, $T$ is of \emph{class} $C_{0\cdot}$ (respectively, $C_{\cdot 0}$) in
 the case that $T$ ($T^*$) is strongly stable, which means $S_T=0$
 ($S_{T^*}=0$), while $T$ is of \emph{class} $C_{00}$ if it is of class
 $C_{0\cdot}$ and of class $C_{\cdot 0}$. We say that $T$ is of \emph{class} $C_{1\cdot}$
(respectively, $C_{\cdot 1}$) if $T^nh \nrightarrow 0$ (respectively $T^{*n}h \nrightarrow 0$)
for all $0\neq h\in \h$. Also, $T$ is
of \emph{class} $C_{11}$ if both $T$ and $T^*$ are of class $C_{1\cdot}$. For two subsets $M$ and $M'$ of $\h$
we write $M\vee M'$ for the smallest closed subspace of $\h$ containing $M\cup M'$.

 A $\B$-valued \emph{semi-spectral measure} on $\T$  is a map $F$ from the $\sigma$-algebra of Borel subsets of $\T$ into $\B$ with the
property that for any $h\in \h$ the map $\sigma\mapsto \langle F(\sigma)h, h\rangle$ is a positive measure
on $\T$. For each contraction $T\in\B$  there exists a
unique $\B$-valued \emph{semi-spectral measure} $F_T$ on $\T$
 satisfying
 \[
 \langle p(T) h, k\rangle = \int_{\T} p(\lambda) d\langle F_T(\lambda) h, k\rangle
 \]
for all $h,k\in \h$ and $p$ a trigonometric polynomial. If $T$ is unitary then $F_T$ is precisely its spectral measure,
denoted also by $E_T$, while for $T=0$ the corresponding $F_0$ is $mI$ where $m$ is the normalized Lebesgue measure on $\T$.

According to \cite{S1} we say that $T$ is \emph{Harnack
dominated} by $T'$ (notation $T \Prec T'$) if there exists a positive constant $c \ge 1$ such that for any analytic polynomial $p$
verifying $\Re p(z) \ge 0$ for $|z| \le 1$ we have
\begin{equation}\label{eq:definition of Harnack}
\Re p(T) \le c \Re p(T').
\end{equation}
We say that $T$ is Harnack
dominated by $T'$ \emph{with constant} $c$ whenever we want to emphasize the constant. We say that $T$ and $T'$ are
\emph{Harnack equivalent} if
$T \Prec T'$ and $T' \Prec T$; we also say in this case that $T$ and $T'$ belong to the same \emph{Harnack part}.
It was proved in \cite{KSS} that the Harnack part of $T$ is formed by $\{T\}$ alone if and only if $T$ is
an isometry or a coisometry (the adjoint of an isometry).

$T$ is said to be \emph{maximal} for the Harnack domination if $T\Prec T'$ implies $T'=T$, and \emph{minimal} if
$T'\Prec T$ implies $T'=T$. Since maximal and minimal elements are Harnack equivalent only with themselves, it follows that they have to be
isometries or coisometries.

Several useful equivalent definitions of the Harnack domination are collected in the following known result.

\begin{theorem}[\cite{S1,S2,S3,SV,AST}]\label{thm:1.1}
 For two contractions $T,T' \in \B$ and $c\ge 1$ and with
the previous notation, the following statements are equivalent:
\begin{itemize}
 \item[(i)] $T$ is Harnack
dominated by $T'$ with constant $c^2$;
\item[(ii)] $K(T,\lambda) \le c^2 K(T',\lambda)$ for every $\lambda \in \D$;
\item[(iii)] for every finite set of vectors $\{h_k\}$ in $\h$ we have
\begin{align}\label{ec13}
\sum_{i,j} \langle T^{[i-j]}h_i, h_j\rangle \le c^2 \sum_{i,j}
\langle T^{'[i-j]}h_i,h_j\rangle.
 \end{align}
\item[(iv)] for every finite set of vectors $\{h_k\}$ in $\h$ we have
 \begin{align}\label{ec12}
\sum_{i,j} \langle V^ih_i, V^jh_j \rangle \le c^2 \sum_{i,j} \langle
V^{'i}h_i, V^{'j} h_j \rangle.
 \end{align}
\item[(v)] there is an operator $A \in \mathcal{B}(\ka',\ka)$ such that $A(\h) \subseteq \h$,
$A\mid \h = I$, $AV' = VA$ and $\|A\| \le c$.
\item[(vi)] The semi-spectral measures of $T,T'$ satisfy $F_T\le c^2 F_{T'}$.
\end{itemize}
\end{theorem}

The next lemma gives simple properties of Harnack domination that we will use in the sequel.

\begin{lemma}\label{le:simple properties of Harnack domination}
	Suppose $T\Prec T'$. Then:
	\begin{itemize}
		\item[(i)] $T^n$ is Harnack	dominated by $T^{'n}$, for any integer $n\ge 2$.
		\item[(ii)] $T_1\Prec T_1'$ and $T_2\Prec T_2'$ if and only if $T_1\oplus T_2\Prec T_1'\oplus T_2'$.
		\item[(iii)] If $\h'\subset\h$ is a closed subspace invariant both to $T$ and $T'$, then $T|\h'\Prec T'|\h'$.
                \item[(iv)] The adjoint $T^*$ of $T$ is Harnack
		dominated by the adjoint of $T^{'}$.
	\end{itemize}
\end{lemma}

\begin{proof}
	The assertions in (i) and (ii) are immediate. As for (iii), note that~\eqref{eq:definition of Harnack}
means that for any $h\in\h$ and polynomial $p$ such that $\Re p \ge 0$ on $\D$ we have
	\[
	\Re \langle p(T)h, h\rangle \le c \Re  \langle p(T')h, h\rangle.
	\]
	The left hand side of the inequality depends on $\langle T^n h, h\rangle $
and $\langle T^*{}^n h, h\rangle = \langle  h, T^n h\rangle  $;  and similarly for
the right hand side. It is then clear that the inequality is satisfied if we take only $h\in \h'$.
The condition (iv) follows easily from
Theorem \ref{thm:1.1}, (ii) (or (iii)).
\end{proof}

Another domination relation, introduced in~\cite{C}, has been used in \cite{AST}. As in the latter,
we say that $T$ is Z-\emph{dominated} by $T'$, and we write $T\PrecZ T'$,
if there exists a bounded operator $\tilde{A}$ from
$\h \vee V'\h$ to $\h \vee V\h$ such that for any $h_0,h_1 \in \h$,
$$ \tilde{A}(h_0+V'h_1) = h_0+Vh_1 .$$
In this case, the operator $\tilde{A}$ is the unique bounded operator from $\h \vee V'\h$
to $\h \vee V\h$ which intertwines $V'$ and $V$
and whose restriction to $\h$ is the identity operator.  We say that $T$ is Z-dominated by $T'$
with constant $c\ge 1$ if $\|\tilde{A}\| \le c$.

\begin{theorem}[cf. Lemma 1 in \cite{AST}]\label{thm:1.2}
With the previous notation, the following  statements are equivalent:
\begin{itemize}
\item[(i)] $T \PrecZ T'$ with constant $c\ge 1$;

\item[(ii)] there is $c' \ge 1$ such that, for any $h\in \h$,
$$ \|D_Th\| \le  c'\|D_{T'}h\| \quad \mbox{ and } \quad \|(T'-T)h\| \le c'\|D_{T'}h\| .$$
\end{itemize}
\end{theorem}

The next corollary follows easily, and shows in particular that isometries are minimal elements for the Z-domination.
\begin{corollary}\label{co:simple z dominations}
\begin{itemize}
	\item[(i)] If $T'$ is an isometry, then $T\PrecZ T'$ if and only if $T=T'$.
	\item[(ii)] If $T$ is an isometry, then $T \PrecZ T'$ if and only if $\|(T'-T)h\| \le c'\|D_{T'}h\|$.
	\item[(iii)] If $T,T'$ are orthogonal projections, then $T \PrecZ T'$ if and only if $T'\le T$.
\end{itemize}
\end{corollary}

It is clear from the
characterization of Harnack domination given by Theorem~\ref{thm:1.1}, (v),
that $T\Prec T'$ implies $T\PrecZ T'$ (with the same constant).
The relation between them is completed by the following result.

\begin{theorem}[cf. Theorem 3 in \cite{AST}]\label{thm:1.3}
If $T\Prec T'$ with constant $c\ge 1$,
then $T_{\lambda} \Prec T'_{\lambda}$ with constant $c$,
for each $\lambda \in \D$, and so $T_{\lambda} \PrecZ T'_{\lambda}$ with constant $c$,
for each $\lambda \in \D$. Conversely, if $T_{\lambda} \PrecZ T'_{\lambda}$
with constant $c'\ge 1$, for each $\lambda \in \D$, then $T\Prec T'$ with constant $c = \sqrt{3}c'$.
\end{theorem}

In the case of positive contractions, there is a closer relation between our two domination relations.
The next result is a consequence of~\cite{KSS} (more precisely, it
follows from Corollary 2.13, Lemma 2.17, and Corollary 3.3 therein).

\begin{lemma}\label{le:KSS}
	Suppose $A,A'\ge 0$ are contractions. Then:
	\begin{itemize}
		\item[(i)] $A\PrecZ A'$ if and only if $I-A^2\le c(I-A'{}^2)$ for some constant $c$.
		\item[(ii)] $A, A'$ are {\rm Z}-equivalent if and only if they are Harnack equivalent.
	\end{itemize}
\end{lemma}

We end this section with a result that shows that Harnack domination implies a useful resolvent estimate. In
the case of isometries this necessary condition is also sufficient.

\begin{theorem}\label{thm:res}
Let $T,T'$ be contractions in $\B$.
Suppose that $T$ is Harnack dominated by $T'$. Then there is $c>0$ such that for each $h \in \h$ we have
\begin{equation}\label{eq:res estim}
\|(I-\lambda T)^{-1}(T-T')h\|^2 \le \frac{c}{1-|\lambda|^2}\|D_{T'}h\|^2 \quad (\lambda \in \D).
\end{equation}
If $T$ is an isometry, the converse is also true: if \eqref{eq:res estim} is satisfied for all $\lambda \in \D$ and $h \in \h$, then
 $T\Prec T'$.
\end{theorem}

\begin{proof}
(i) Suppose that $T$ is Harnack dominated by $T'$ with constant $c$. Thus $T_\lambda \PrecZ T'_{\lambda}$
with constant $c$ for every $\lambda \in \D$. Then
\begin{equation}\label{eq:myeq1}
 \|(T'_\lambda -T_\lambda)x\|^2 \le c
 \|D_{T'_\lambda}x\|^2
\end{equation}
for each $x \in \h$.

Denoting $h = (I-\overline{\lambda} T')^{-1}x$, we obtain $T'_\lambda x = (T'-\lambda I)h$
and $ \|D_{T'_\lambda}x\|^2= (1-|\lambda|^2)\|D_{T'}h\|^2$.
Since
\begin{equation}\label{eq:T'lambda-Tlambda}
T_\lambda - T'_\lambda = \frac{1-|\lambda|^2}{\overline{\lambda}} \left[ (I-\overline{\lambda} T)^{-1}
- (I-\overline{\lambda} T')^{-1} \right]
\end{equation}
we get from~\eqref{eq:myeq1}
$$  \left\|\left[ (I-\overline{\lambda} T)^{-1} - (I-\overline{\lambda} T')^{-1} \right]x\right\|^2
\le \frac{c|\lambda|^2}{1-|\lambda|^2}\|D_{T'_\lambda}h\|^2 .$$
But
$$ \left[ (I-\overline{\lambda} T)^{-1} - (I-\overline{\lambda} T')^{-1} \right]x
= (I-\overline{\lambda} T)^{-1}\left[ \overline{\lambda}(T-T')\right]h ,$$
and therefore~\eqref{eq:res estim} is true.

Suppose now that $T$ is an isometry and that
~\eqref{eq:res estim} is satisfied
for every $\lambda\in\D$. The above proof  can be reversed to get
$$ \|(T'_\lambda -T_\lambda)x\|^2 \le
c\|D_{T'_\lambda}x\|^2 $$
for each $x \in \h$. Since $T$ is an isometry, the same is true for each M\"obius
transform $T_\lambda$. Therefore $T_\lambda \PrecZ T'_{\lambda}$ uniformly in  $\lambda \in \D$, and thus $T\Prec T'$.
\end{proof}

\begin{remark}
With similar methods it can be proved that
the contraction $T$ is Harnack dominated by $T'$ with constant $c$
if and only if, for each $h \in \h$, and each $\lambda \in\D$ one has
\begin{eqnarray*}
\|(I-\lambda T)^{-1}(T-T')h\|^2  &+& \frac{1}{1-|\lambda|^2}\left(1-\frac{1}{c^2}\right)\left(\|z\|^2-\|Tz\|^2\right) \\
 &\le& \frac{c^2-1}{1-|\lambda|^2}\left(\|h\|^2-\|T'h\|^2\right),
\end{eqnarray*}
where $z = z(\lambda,h) = (I-\lambda T)^{-1}(I-\lambda T')h$. We will not use this more general result in the sequel.
\end{remark}

\begin{corollary}\label{cor:min}
	A contraction is a minimal element for Harnack domination if and only if it is an isometry or a coisometry.
\end{corollary}

\begin{proof}
We have already noticed above that a minimal element has to be an isometry or a coisometry.
Suppose then that $T'$ is an isometry and that $T$ is Harnack dominated by $T'$.
Then the inequality (\ref{eq:res estim}) implies $(I-\lambda T)^{-1}(T-T')h = 0 $ for
each $h \in \h$, and so $T=T'$. Thus $T'$ is minimal.

Using Lemma \ref{le:simple properties of Harnack domination}, (iv), we obtain that
$T^{'\ast}$ is minimal whenever $T'$ is minimal.
Thus coisometries are also  minimal elements for Harnack domination.
\end{proof}

\section{Maximal elements for Harnack domination} \label{Sect:exemplu}

 In this section we prove that singular unitary operators are precisely the
maximal elements with respect to Harnack domination.

Given a finite measure $\mu$ on $\T$ we denote by $D_{\mu}(x)$ its upper density
$$ D_{\mu}(x) = \limsup_{\epsilon \to 0}\frac{\mu{(x-\epsilon,x+\epsilon)}}{2\epsilon}$$
at $x$. It is known that if $\mu$ is singular with respect to Lebesgue measure, then
$D_{\mu}(x) = \infty$ $\mu$ a.e.

\begin{theorem}\label{thm:singular}
	Let $U \in \B$ be a unitary operator with spectral measure $E_U$ and
	let $T\in \B$ be a contraction. Let $h\in \h$.
	Suppose that $U$ is Harnack dominated by
	$T$  and that $y = (U-T)h\not=0$.
	If $\mu_y=\langle E_U y, y\rangle $, then for any $t \in \T$ we have
	$D_{\mu_y}(t) < +\infty$.
\end{theorem}

\begin{proof} If $U\Prec T$ with constant $c$, then the resolvent estimate~\eqref{eq:res estim} is satisfied with the same constant $c$ and thus
\[
\|(I-\lambda U)^{-1}(U-T)h\|^2 \le \frac{c}{1-|\lambda|^2} (\|h\|^2-\|Th\|^2).
\]
	By the spectral theorem, we have
	\[
	\|(I-\lambda U)^{-1}(U-T)h\|^2 = \int_0^{2\pi} \frac{1}{|1-\lambda
		e^{it}|^2} d\mu_y(t)
	\]
	for every $\lambda \in \D$. Let $\epsilon > 0$ and fix $t_0\in\T$.
	For $\lambda=(1-\epsilon) e^{-it_0}$, we obtain
	\begin{eqnarray*}
		\int_0^{2\pi} \frac{1}{|1-\lambda e^{it}|^2} d\mu_y(t) &\ge &
		\int_{t_0-\epsilon}^{t_0+\epsilon} \frac{1}{|e^{-it}-(1-\epsilon)
			e^{-it_0} |^2} d\mu_y(t) \\
		&\ge & \frac{1}{2\epsilon^2} \mu_y([t_0-\epsilon,
		t_0+\epsilon]).
	\end{eqnarray*}
	Therefore
	\begin{eqnarray*}
		\frac{\mu_y([t_0-\epsilon,
			t_0+\epsilon])}{2\epsilon^2}&\le &\int_0^{2\pi} \frac{1}{|1-\lambda
			e^{it}|^2} d\mu_y(t)
		\\
		&\le & \frac{c}{1-|\lambda|^2} (\|h\|^2-\|Th\|^2)\\
		&\le&  \frac{c}{\epsilon} \|h\|^2.
	\end{eqnarray*}
	We obtain
	\begin{equation*}\label{eq:contr}
	\frac{\mu_y([t_0-\epsilon,
		t_0+\epsilon])}{2\epsilon} \le c\|h\|^2 ,
	\end{equation*}
	which proves the theorem.
\end{proof}

\begin{corollary}\label{cor:alpha-max}
	Any singular unitary operator  is a maximal
	element for Harnack domination.
\end{corollary}

Note that the particular case of the maximality of a symmetry (a unitary operator $T$ with $T^2=I$) follows from \cite[Corollary 3.3 and Proposition 3.5]{KSS}.

The next lemma, which we need here as well as in Section~\ref{sect:6}, is a simple computation. We use the notation $\xi\otimes \eta$ for
the rank one operator $x\mapsto \langle x, \eta\rangle \xi$.

\begin{lemma}\label{le:one dim pert of isometries}
	Suppose $U\in\B$ is an isometry, $\xi\in\h$, $\|\xi\|=1$, and $\alpha\in\C$. If   $T=U-(1-\alpha) U\xi\otimes \xi$, then
	\begin{equation}\label{eq:one dim pert iso}
	 I-T^*T = (1-|\alpha|^2) \xi\otimes\xi.
	\end{equation}
	Consequently, $\|T\|\le 1$ if and only if $|\alpha|\le 1$, in which case $D_T^2$ is given by~\eqref{eq:one dim pert iso}.
\end{lemma}

\begin{theorem}\label{th:abs cont unitaries are dominated}
(i)	If $U\in \B$ is an absolutely continuous unitary operator, then $U$
is not maximal with respect to Harnack domination.

(ii) A unilateral shift of arbitrary multiplicity is not maximal with respect to Harnack domination.

\end{theorem}

\begin{proof}
	(i) Applying the spectral theorem and Lemma \ref{le:simple properties of Harnack domination}, (ii),
	we may assume that $U$ is the operator of
multiplication by
the variable $\zeta$ on $L^2(\omega, d\nu)$, where $\omega\subset [0, 2\pi]$ and $d\nu$ is
Lebesgue measure normalized so as to have $\nu(\omega)=1$. If $\xi=1$ (the constant function),
then taking $\alpha=0$ in Lemma~\ref{le:one dim pert of isometries} it follows that the operator $T=U-U\xi\otimes \xi$
is a contraction, while~\eqref{eq:one dim pert iso} yields
	\begin{equation}\label{eq:abs cont 1}
	\|D_T f\|^2= \langle D_T^2f, f\rangle= |\langle f,1\rangle|^2
	\end{equation}
	for all $f\in L^2(\omega, d\nu)$.
	Also, since $(U-T)f=\langle f,1\rangle e^{it}$, we have for such an $f$ and $|\lambda|<1$
	\begin{equation}\label{eq:abs cont 2}
	\begin{split}
	\| (I-\lambda U)^{-1} (U-T)f\|^2 &=
	\|(I-\lambda U)^{-1} \langle f,1\rangle 1\|^2\\
&	= \int_\omega \frac{|\langle f,1\rangle|^2}{|1-\lambda e^{it}|^2}  d\nu(t).
	\end{split}
	\end{equation}
	Since
	\[
	\begin{split}
		\int_\omega \frac{1}{|1-\lambda e^{it}|^2}  d\nu(t)
	&=\frac{1}{m(\omega)}\int_\omega \frac{1}{|1-\lambda e^{it}|^2}  dm(t)\\& \le \frac{1}{m(\omega)}\int_{[0,2\pi]} \frac{1}{|1-\lambda e^{it}|^2}  dm(t) = \frac{1}{m(\omega)(1-|\lambda|^2)},
	\end{split}
\]
we obtain by~\eqref{eq:abs cont 1} and~\eqref{eq:abs cont 2},
\[
\| (I-\lambda U)^{-1} (U-T)f\|^2 \le \frac{1}{m(\omega)(1-|\lambda|^2)} 	\|D_T f\|^2.
\]
By Theorem~\ref{thm:res}, it follows that $U$ is Harnack dominated by $T$, and is therefore not maximal. 	

(ii) By Lemma~\ref{le:simple properties of Harnack domination} (ii) it is enough to show the non-maximality of
the unilateral shift of multiplicity one, which is unitarily equivalent to the restriction
to  $H^2$ of the unitary operator $U$ defined as multiplication by the
variable $\zeta$ acting on $L^2([0,2\pi], dm)$. In the first part of the proof
we have shown that $U$ is Harnack dominated by $T=U-U\xi\otimes \xi$, where $\xi$ is the
constant function. Since $U\xi\in H^2$, $T H^2\subset H^2$. Therefore the assertion follows
from  Lemma~\ref{le:simple properties of Harnack domination} (iii).
\end{proof}

We can give now the promised characterization of  elements maximal with respect to Harnack domination.

\begin{theorem}\label{th:charact maximal elements}
A contraction $T\in \B$ is a maximal element with respect to Harnack domination if and only if it is a singular unitary operator.	
\end{theorem}

\begin{proof}
	Suppose $T\in \B$ is maximal with respect to Harnack domination. In particular, it follows that
the Harnack equivalence class containing $T$ is reduced to $\{T\}$, whence it
follows by~\cite[Corollary 3.4]{KSS} that $T$ is an isometry or a coisometry.
Since $T$ is maximal if and only if $T^*$ is maximal, we may assume that $T$ is an isometry.
	
	By the Wold decomposition, we can write $T=S\oplus U$, where $S$ is a unilateral
shift of some multiplicity and $U$ is unitary. By
Theorem~\ref{th:abs cont unitaries are dominated} (ii) $S$ cannot appear,
and thus $T$ has to be unitary. Then the assertion follows from Corollary~\ref{cor:alpha-max}
and Theorem~\ref{th:abs cont unitaries are dominated} (i).
\end{proof}

\section{Ergodic properties and spectrum}

An interesting feature of Harnack domination of contractions is the way it implies preservation of certain properties.
The results of this section show, in particular,  that this is true about
the peripheral spectrum. Our development will go through establishing some ergodic properties.

The following lemma is proved in~\cite[Theorem 3.1]{KSS}.

\begin{lemma}\label{le:KSS row matrix}
	Let $T$ and $T'$ be contractions on $\h$ such that
	$T\Prec T'$. If $C$ denotes the bounded linear operator defined by
	\begin{equation}\label{eq:definition of C}
	CD_{T'}h=(T-T')h, \quad h\in \h,
	\end{equation}
	then the linear operator $X:l_{\N}^2 (\mathcal{D}_{T'})\to \h$ having the row
	matrix representation
	$$
	X=[C,TC,T^2C, ...]
	$$
	is also bounded.
\end{lemma}

Note that the boundedness of $C$ is given by Theorem~\ref{thm:1.2}.

\begin{theorem}\label{te39}
	Let $T$ and $T'$ be contractions on $\h$ such that
	$T\Prec T'$. Then:
	
	\begin{itemize}
		\item[(i)] $\n(I-T)=\n(I-T')$ and $\overline{\R(T-T')}\subset \overline{\R(I-T)}=\overline{\R(I-T')}$.
		\item[(ii)]
		With respect to the decomposition $\h=\n(I-T)\oplus \overline{\R(I-T)}$ we have
		\begin{equation}\label{eq:decompositions I oplus T_1}
			T=I\oplus T_1,\quad T'=I\oplus T_1',
		\end{equation}
		and $T_1\Prec T_1'$.
		\item[(iii)]  For every sequence $\{\alpha _n\}\subset l_{\N}^2(\C)$ the series
		$$
		\sum_{n=0}^{\infty} \alpha _n T^n (T-T')h
		$$
		converges in norm, for every $h\in \h$.
		
	\end{itemize}
	
\end{theorem}

\begin{proof}
	(i)
		It follows immediately from Theorem~\ref{thm:1.2} that  $T=T'=I$ on $\n(I-T')$, hence
		$\n(I-T')\subset \n(I-T)$.
	
	For the opposite inclusion, note that, if $C$ is defined by~\eqref{eq:definition of C},
then by Lemma~\ref{le:KSS row matrix} it follows, in particular, that for $h\in
	\h$,
	$$
	\|C^*T^{*n}h\|\to 0, \quad n\to \infty.
	$$
	This implies that
	\[
	\n(I-T)=\n(I-T^*)\subset \n(C^*)=\n(T^*-T'^{*}).
	\]
	If $h\in \n(I-T)$, then
	$C^*h=0=(T^*-T'^{*})h$, hence $T'^{*}h=T^*h=h$, that is $h\in
	\n(I-T'^{*})=\n(I-T')$. Therefore $\n(I-T)\subset \n(I-T')$.
	
 Consequently,
	$\n(I-T)=\n(I-T')$, which also implies that
$$
	\overline{\R(T-T')}\subset \mathcal{D}_{T'^{*}} \subset \overline{\R(I-T')}=\overline{\R(I-T)}.
	$$
	Here the first inclusion follows by Theorem \ref{thm:1.2} (ii) from the relation $T^* \Prec T'^*$, while the second inclusion is true
because $\overline{\R(I-T')}^{\perp}=\n(I-T')\subset \n(I-T'T'^{*})=\n(D_{T'^{*}}).$ The last equality is true since $\n(I-T^{*})=\n(I-T^{'*})$ and the same is true for their orthogonal complements. 
	
	(ii) The decompositions in direct sum are an immediate consequence of the contractivity of $T$ and $T'$.
The Harnack domination $T_1\Prec T_1'$ follows then from Lemma~\ref{le:simple properties of Harnack domination} (ii).
	
	(iii) 	 The boundedness of $X$ means, in particular, that for every
	$d=\{d_n\}_{n\in \N} \in l_{\N}^2(\mathcal{D}_{T'})$ the series
	$Xd=\sum_{n=0}^{\infty} T^nCd_n$ converges in the norm of $\h$. Thus, if $\{\alpha _n\} \subset l_{\N}^2(\C)$
then setting $d_n=\alpha _n D_{T'}h$ for $h\in \h$ one obtains that the series
	$$
	\sum_{n=0}^{\infty} \alpha _n T^n (T-T')h
	$$
	converges in norm, for every $h\in \h$.
\end{proof}

A first application of Theorem~\ref{te39} is related to functional calculus.
 Lemma 2.2 of \cite{GHT} states that if $f(z)=\sum_{n=0}^{\infty} \alpha _n z^n$ is an analytic function on $\D$
which has no zeroes in $\D$ and  such that the function $\frac{1}{f}$ has absolutely summable
Taylor coefficients, then, whenever $T$ is a contraction on $\h$ and   $x\in\h$ is
such that $y:=\sum_{n=0}^{\infty} \alpha _n T^n x$ converges weakly, we have $\frac{1}{f}(T)y=x$.  If  $T \Prec T'$,
then Theorem~\ref{te39} (iii) produces a whole class of such vectors $x$, namely those in $\R(T-T')$.
Therefore $\R(T-T')\subset \R(\frac{1}{f}(T))$.

More interesting applications of Theorem~\ref{te39} are related to the \emph{Ces\`aro means} of a contraction $T$. These are defined by
\[
M_n(T)=\frac{1}{n+1} \sum_{j=0}^n T^j.
\]
It is known that $\{M_n(T)\}$ uniformly converges in $\B$ if and only if $\R(I-T)$ is closed (see \cite{L}),
and such a contraction is called {\it uniformly Ces\`aro ergodic}. It is also known (see \cite{Kr}) that if the Ces\`aro means $\{M_n(T)\}$ weakly converge in $\B$, then its limit is the ergodic projection $P_T$, that is the orthogonal projection onto $\n(I-T)$. So, by the decomposition (\ref{eq:decompositions I oplus T_1}) we have $M_n(T)-P_T=0 \oplus M_n(T_1)$ and $T^n-P_T=0\oplus T_1^n$ on $\h=\n(I-T) \oplus \overline{\R(I-T)}$. We have thus proved the following lemma.

\begin{lemma}\label{le:convergence to 0}
	If $T$ is a contraction on $\h$, then any type (weak, strong or uniform) convergence of $\{M_n(T)\}$, respectively of $\{T^n\}$, is equivalent with the corresponding convergence to 0  of $\{M_n(T_1)\}$, or $\{T_1^n\}$ respectively.
\end{lemma}

A related notion is
the \emph{one-sided ergodic Hilbert transform} of  $T$, which is given by the formula
\begin{equation}\label{eq:hilbert transform}
H_Tx:=\sum_{n=1}^{\infty} \frac{T^n}{n}x,
\end{equation}
having as domain  the subspace $\Dom H_T$ of vectors $x\in\h$ for which the series in~\eqref{eq:hilbert transform}
is norm convergent. We refer the reader to~\cite{CCL}, where it is also proved that
\[
\R(I-T)\subset \Dom H_T\subset \overline{\R(I-T)}.
\]
It was shown in~\cite[Theorem 4.1]{GHT} that if $x\in\Dom H_T$, then $(\log n)M_n(T)x \to 0$ when $n\to\infty$.

Using Theorem~\ref{te39}, we get the following relationship between the ranges of $I-T$ and $I-T'$ when $T$ is Harnack dominated by $T'$.

\begin{corollary}\label{co:hilbert transform}
	Suppose that $T$ and $T'$ are contractions on $\h$ and $T \Prec T'$. Then
	\[
	\R(I-T)=\R(T-T')+\R(I-T')\subset \Dom H_T.
	\]
	
	In particular, if $T$ and $T'$ are Harnack equivalent then $\R(I-T)=\R(I-T')$.
\end{corollary}

\begin{proof}

We can apply the above remark concerning the functional calculus by choosing the function $f(z)=(1-z)^{-1}$ for $z\in \D$ to conclude that $\R(T-T') \subset \R(I-T)$. This later implies $\R(I-T')\subset \R(I-T)$, and also $\R(T-T')+\R(I-T') \subset \R(I-T)$, while the reverse inclusion is trivial. We obtain the inclusion quoted in corollary. When $T$ and $T'$ are Harnack equivalent we have by symmetry $\R(I-T)=\R(I-T')$.
\end{proof}

These ergodic properties may be used to relate Harnack domination to the spectrum of contractions. Note first that
the following lemma is implicitely proved in~\cite[Theorem 1]{AST}.
As usually $\sigma(T)$ denotes the spectrum of $T$ and $\sigma_p(T)$ its point spectrum.

\begin{lemma}\label{le:AST theorem 1}
	If $T \PrecZ T'$, then $\sigma(T')\cap \T\subset\sigma(T)\cap \T $.
\end{lemma}

\begin{theorem}\label{pr812}
	Let $T,T'$ be contractions on $\h$ such that $T \Prec T'$.
Then $\sigma (T) \cap \T = \sigma(T')\cap \T$ and $\sigma _p(T)\cap \T=\sigma_p(T')\cap \T$.
In particular, $\sigma (T) \subset \D$ if and only if $\sigma (T') \subset \D$.

\end{theorem}

\begin{proof}
	(i) Let $\lambda \in \T$ be such that $\lambda \notin \sigma (T')$. Since $T \Prec T'$ we
have also $\overline{\lambda} T \Prec \overline{\lambda}T'$ (by Theorem \ref{thm:1.1} (ii), for instance).
Thus, by Corollary \ref{co:hilbert transform} we have
	$$
	\h=\R(I-\overline{\lambda}T')=\Dom H_{\overline{\lambda }T},
	$$
	and so $(\log n)M_n(\overline{\lambda }T)x \to 0$ for every $x\in \h$.
	According to the uniform boundedness principle, $(\log n)M_n(\overline{\lambda }T)$ is
bounded in norm, and so $\|M_n(\overline{\lambda }T)\|$ tends to $0$ when $n$ tends to infinity.
	But this implies that $I-\overline{\lambda}T$ is invertible, hence $\lambda \notin \sigma(T)$.
Thus, $\sigma (T) \cap \T \subset \sigma(T')\cap \T$. The opposite inclusion follows from Lemma~\ref{le:AST theorem 1}.

	By Theorem \ref{te39} we have $\n(\lambda I-T)=\n(\lambda I-T')$ for each $\lambda \in \T$,
which means $\sigma _p(T) \cap \T =\sigma _p (T') \cap \T$. 	
\end{proof}

\begin{corollary}\label{co:uniformly ergodic}
	Let $T,T'$ be contractions on $\h$ such that $T \Prec T'$. Then $T$ is uniformly Ces\`aro
ergodic if and only if $T'$ is uniformly Ces\`aro ergodic. Also,
$\{T^n\}$ uniformly converges if and only if $\{T'^n\}$ uniformly converges.
\end{corollary}

\begin{proof} As noted above, $T$ is uniformly ergodic if and only if $\R(I-T)$ is closed.
	Using the decompositions~\eqref{eq:decompositions I oplus T_1} it is easy to see
that $\R(I-T)=\R(I-T_1)$, $\R(I-T')=\R(I-T'_1)$, $I-T_1$ and $I-T_1'$ are injective, and $T_1\Prec T_1'$.
Therefore $\R(I-T)$ is closed if and only if $I-T_1$ is invertible, and similarly for $T'$.
But from Theorem~\ref{pr812} it follows that $1\in\sigma(T_1)$ if and only
if $1\in\sigma(T'_1)$, which proves the statement.

For the second statement, it follows from Lemma~\ref{le:convergence to 0} that $\{T^n\}$ converges in $\B$ if and only if  $\|T^n_1\| \to 0$, and the last assertion is equivalent to $\sigma(T_1) \subset \D$. The same being true about $T'$, the proof is finished by applying Theorem~\ref{pr812}.
\end{proof}

Another consequence of Theorem \ref{pr812} is related to the Katznelson-Tzafriri theorem \cite{KT} which implies that
for a contraction $T\in\B$ we have $\sigma(T) \subset \D \cup \{1\}$ if and only if $\|T^n(T-I)\| \to 0$ as $n\to \infty$.
So we obtain the following
\begin{corollary}\label{co47}
Let $T,T'$ be contractions on $\h$ such that $T \Prec T'$. Then $\|T^n(T-I)\|\to 0$ if and only if $\|T'^n(T'-I)\|\to 0$.
\end{corollary}

 \section{ Harnack domination and various classes of contractions}
 \medskip

In this section we intend to show that certain clases of contractions are preserved by Harnack domination.
This will be used, in particular, to give an alternate proof of Corollary~\ref{cor:alpha-max}.
The main tool used is the asymptotic limit of contractions.

\begin{lemma}\label{pr11}
Let $T$ and $T'$ be two contractions on $\h$ such that $T$ is
Harnack dominated by $T'$. The following statements hold :

\begin{itemize}
	\item[(i)]
	There exists a constant $c\ge 1$ such that
	\begin{align}\label{ec14}
	\frac{1}{4} \mid \langle (S_T-S_{T'})h,h \rangle \mid
	^2+\|(I-S_T)^{1/2}h\|^2 \le c^2 \|(I-S_{T'})^{1/2}h\|^2,
	\end{align}
	for all $h\in \h$.
	\item[(ii)]
	 We have $\n(I-S_{T'})\subset \n(I-S_T)$ and $T=T'$ on $\n(I-S_{T'})$.
	 \item[(iii)]  $S_T^{1/2}$ is Z-dominated by $S_{T'}^{1/2}$.
	 \item[(iv)] If, moreover,
	 $S_T^{1/2}$ and $S_{T'}^{1/2}$ are Z-equivalent then they
	 are Harnack equivalent and
	 $\n(I-S_T)=\n(I-S_{T'})$.
\end{itemize}

 \end{lemma}

 \begin{proof}
Suppose that $T \Prec T'$ with constant $c\ge 1$. Then $T^n \Prec
T^{'n}$ with constant $c$ and, in particular, $T^n\PrecZ T^{'n}$
with the same constant $c$, for every $n\ge1$. This implies by
Theorem~\ref{thm:1.2} that for $h\in \h$,
$$
\|(T^n-T^{'n})h\|^2 +\|D_{T^n}h\|^2\le c^2 \|D_{T^{'n}h}\|^2.
$$
Therefore
$$
\mid \|T^nh\|-\|T^{'n}h\| \mid ^2+\langle (I-T^{*n}T^n)h,h \rangle
\le c^2 \langle (I-T^{'*n}T^{'n})h,h\rangle,
$$
and letting $n\to \infty$ we get

\begin{equation}\label{eq:my4.1}
	\mid \|S_T^{1/2}h\|-\|S_{T'}^{1/2}h\| \mid ^2 +
	\|(I-S_T)^{1/2}h\|^2\le c^2 \|(I-S_{T'})^{1/2}h\|^2.
\end{equation}

Now, if $\|h\|=1$ we have
$$
\mid \langle (S_T-S_{T'})h,h\rangle \mid =\mid
\|S_T^{1/2}h\|^2-\|S_{T'}^{1/2}h\|^2 \mid \le 2 \mid
\|S_T^{1/2}h\|-\|S_{T'}^{1/2}h\|\mid ,
$$
which, together with~\eqref{eq:my4.1}, yields~\eqref{ec14}.

From \eqref{ec14} it follows immediately that $\n(I-S_{T'})\subset
\n(I-S_T)$.  Since $\n(I-S_{T'}) \subset \n(D_{T'}) \subset
\n(T-T')$ (the last inclusion follows from Theorem~\ref{thm:1.2}), we conclude that $T=T'$ on $\n(I-S_{T'})$.

Inequality (\ref{ec14}) also implies $\|D_{S_T^{1/2}}h\| \le
c \| D_{S_{T'}^{1/2}}h\|$ for $h\in \h$. Since $S_T$, $S_{T'}$
are positive contractions, Lemma~\ref{le:KSS} (i) implies that $S_T^{1/2}\PrecZ S_{T'}^{1/2}$.

If $S_T^{1/2}$ and $S_{T'}^{1/2}$ are Z-equivalent, then Lemma~\ref{le:KSS} (ii) implies that they are Harnack equivalent.
Therefore their squares $S_T$ and $S_{T'}$ are Harnack
equivalent, and so $\n(I-S_T)=\n(I-S_{T'})$.
 \end{proof}

 \begin{corollary}\label{co12}
If $T$ and $T'$ are Harnack equivalent contractions on $\h$ then
$S_T^{1/2}$ and $S_{T'}^{1/2}$, as well as $S_T$ and $S_{T'}$, are
Harnack equivalent. In this case one has $\n(I-S_T)=\n(I-S_{T'})$
and $\n(I-S_{T^*})=\n(I-S_{T^{'*}})$.
 \end{corollary}

\begin{remark}\label{re13}
If $T\Prec T'$ but they are not Harnack equivalent, then
$\n(I-S_{T'}) \subsetneqq \n(I-S_T)$, in general. An example may be obtained by taking
$T$ to be an absolutely continuous unitary and $T'$ a nonisometric contraction that
dominates $T$, as in the proof of Theorem~\ref{th:abs cont unitaries are dominated}.
\end{remark}

\begin{lemma}\label{le:domination and the unitary part}
	Let $T$ and $T'$ be two contractions on $\h$ such that $T$ is
	Harnack dominated by $T'$. If $\h_u$, $\h_u'$ are the maximum subspaces of $\h$ which
	reduce $T, T'$ to unitary operators, respectively, then
	$\h_u'\subset \h_u$, $\h_u'$ reduces $T$ and $T|\h_u'=T'|\h_u'$, while  $T|_{\h \ominus \h_u'}$
	is Harnack dominated by $T'|_{\h\ominus \h_u'}$.
\end{lemma}

\begin{proof}
	Since
	$T\Prec T'$ and (by
Lemma~\ref{le:simple properties of Harnack domination}, (iv)) $T^* \Prec T^{'*}$,
we have by  Lemma~\ref{pr11} $\n(I-S_{T'}) \subset \n(I-S_T)$ and
	$\n(I-S_{T^{'*}}) \subset \n(I-S_{T^*})$. Therefore
	\[
	\h_u'=\n(I-S_{T'})\cap \n(I-S_{T^{'*}}) \subset \n(I-S_T)\cap \n(I-S_{T^*})=
	\h_u.
	\]
	In addition, as $T=T'$ on $\n(I-S_{T'})$ and $T^*=T^{'*}$ on
	$\n(I-S_{T^{'*}})$, it follows that $\h_u'$ reduces $T$ to a unitary
	operator. Hence $\h\ominus \h_u'$ also reduces $T$ and $T'$, while
	by (\ref{ec13}) we have $T|_{\h \ominus \h_u'}\Prec T'|_{\h \ominus
		\h_u'}$.
\end{proof}

The next theorem gathers a series of results that show how different classes of contractions
behave with respect to Harnack domination.

\begin{theorem}\label{pr14}
Let $T$ and $T'$ be two contractions on $\h$ such that $T$ is
Harnack dominated by $T'$.
\begin{itemize}
	\item[(i)] If $T$ is completely nonunitary then $T'$ is also completely nonunitary.
	\item[(ii)]
	$T$ is absolutely continuous if and only if $T'$ is absolutely continuous.
	\item[(iii)] $T$ belongs to the class $C_{0\cdot}$, $C_{\cdot 0}$, or
	$C_{00}$ if and only if  $T'$ belongs to the same class, respectively.
	\item[(iv)] $T^n$ is strongly convergent if and only if $T'{}^n$ is strongly convergent.
	\item[(v)]
	$T$ is weakly stable if and only if  $T'$ is weakly stable.
    \item[(vi)] $T^n$ is weakly convergent if and only if $T^{'n}$ is weakly convergent.
\end{itemize}

\end{theorem}

\begin{proof} By Lemma~\ref{le:domination and the unitary part}, if $T$ is
completely nonunitary then $T'$ is completely nonunitary. If $T'$ is not absolutely continuous,
it should have a reducing subspace on which it is a singular unitary operator. But,
again by Lemma~\ref{le:domination and the unitary part}, $T$ would have the same
property (since it coincides with $T'$ on the space on which the latter is unitary).

  On the other hand, if $T'$ is absolutely continuous, then the $\B$-valued semi-spectral
measure of $T'$ is absolutely continuous with respect to Lebesgue measure. By Theorem~\ref{thm:1.1}
the same is true about the $\B$-valued semi-spectral measure of $T$, and thus  $T$ is absolutely continuous.
We have thus proved (i) and (ii).

 It is enough to prove (iii) for the case $C_{0\cdot}$ (we may consider adjoints in the other cases).
Assume first that $T$ is of class $C_{0\cdot}$, that is $S_T=0$. From Lemma~\ref{pr11} (iii)
it follows that $0\PrecZ S_{T'}^{1/2}$. By~\cite[Corollary 2]{AST}  we have $\|S_{T'}^{1/2}\| < 1$.
This forces $S_{T'}=0$, that is $T'$
is of class $C_{0\cdot}$.

Conversely, suppose $T'$ is of class $C_{\cdot 0}$, that is $T^{'*n} \to 0$ strongly on $\h$.
This means (see \cite{SFb}) that if $V'$ on $\ka '$ is the minimal isometric dilation of $T'$ then $V^{'*n}\to 0$ strongly on $\ka '$. If $V$ on $\ka$ is the minimal isometric dilation  of $T$, then $T \Prec T'$ implies that there exists  $A\in \mathcal{B}(\ka ',\ka)$ satisfying $AV'=VA$ such that $A$ is an extension of $I_{\h}$. Then $A^*$ is a lift of $I_{\h}$, that is $P_{\h}A^*k=P_{\h}k$, for each $k\in \ka$. Therefore, for any integer $n\ge 1$ and $h\in \h$ we have ($V^*$ being an extension of $T^*$)
$$
T^{*n}h=V^{*n}h=P_{\h}A^*V^{*n}h=P_{\h}V^{'*n}A^*h\to 0, \quad n\to \infty.
$$
Hence $T$ is of class $C_{\cdot 0}$. To close this converse part of (iii), let us remark that if $T'$ is of class $C_{0\cdot}$ then as $T^* \Prec T'^*$, we can apply the previous argument for $T^*$ and $T'^*$ to conclude that $T$ is also of class $C_{0\cdot}$.

Suppose now that $T$ is weakly stable. By the Foguel decomposition
of $T'$ (see \cite[7.2]{K}) we have $\h=\h_1'\oplus \h_0'$, where  $\h_1'$
reduces $T'$ to a unitary operator, and $T'|_{\h_0'}$ is weakly
stable, $\h_0'$ being the maximum subspace of $\h$ with this
property. So $\h_1'\subset \h_u'$, and by Lemma~\ref{le:domination and the unitary part}  we have
$T=T'$ on $\h_1'$, hence $\h_1'$ is invariant for $T$. Since $T$ is
weakly stable on $\h$, it follows that $T'$ is weakly stable on
$\h_1'$, therefore $\h_1'=\{0\}$. We conclude that $T'$ is weakly
stable on $\h=\h_0'$.

Conversely, if we suppose that $T'$ is weakly stable,
then its unitary part $T'|\h_u'$ is weakly stable,
and $T=T'$ on $\h'_u$ by Lemma~\ref{le:domination and the unitary part}. In addition, $T|_{\h \ominus \h_u'} \Prec T'|_{\h \ominus \h_u'}$ while the contraction in the right side is completely nonunitary. By the above statement $(ii)$ both these contractions are absolutely continuous, hence $T=T|_{\h_u'} \oplus T|_{\h \ominus \h_u'}$ is weakly stable.

Finally, (iv) follows from (iii) and (vi) follows from (v) by applying Lemma~\ref{le:convergence to 0}.
\end{proof}

\begin{remark}\label{re56}
	The implication in Theorem~\ref{pr14} (i) cannot be reversed;
indeed, we have seen in Theorem~\ref{th:abs cont unitaries are dominated}
that a unitary operator can be Harnack dominated by a completely nonunitary
contraction.
\end{remark}

\begin{remark}\label{re57}
The weak convergence mentioned in Theorem \ref{pr14} (vi) is
equivalent to the fact that the contraction $T$ has the \emph{Blum--Hanson property} \cite{BH},
which means that for every subsequence $\{k_n\}$ of positive integers
and each $h\in \h$ the sequence $\{\frac{1}{N}\sum_{n=1}^N T^{k_n}h\}$
converges in the norm topology (see for instance \cite{JK}). 
So (vi) above can be reformulated as: $T$ has the
Blum--Hanson property if and only if $T'$ has the same property. Note that for isometries induced by measure-preserving transformations, the Blum--Hanson property is equivalent to the strong mixing property of the transformation (see also~\cite{FEHN} for other related results).
\end{remark}

\begin{remark}\label{re58}
	A consequence is the following  alternate proof of Corollary~\ref{cor:alpha-max}.
	Suppose $U$ is a singular unitary operator, $T$ is a contraction on $\h$,
	and $U\Prec T$.
	If we denote by $\h_u$ the maximal space that reduces $T$ to a unitary,
then by Lemma~\ref{le:domination and the unitary part}  we have $U_1:=
	U|_{\h\ominus \h_u} \Prec T|_{\h\ominus \h_u}=:T_1$.
	By Theorem~\ref{pr14} $U_1$ is absolutely continuous, which implies $\h\ominus \h_u=\{0\}$.
Therefore $T$ is unitary, whence $U=T$.
\end{remark}

\section{Examples and counterexamples}\label{sect:6}

We give in this section several examples showing the usefulness of the resolvent
estimate and the existence of some spectral and structural properties
which are not preserved by Harnack domination.

\begin{example}
In this example $S$ denotes the shift operator of multiplicity $\dim \e$: for
$x = (x_0,x_1, x_2, \cdots ) \in \ell^2_{\N}(\e)$ we set
$$ S(x_0,x_1, x_2, \cdots ) = (0, x_0, x_1, x_2, ...). $$
Let $A \in\mathcal B(\e)$ and consider the operator $T'$ defined on $\ell^2_{\N}(\e)$ by
$$ T' (x_0,x_1, x_2, \cdots ) = (0, Ax_0, x_1, x_2, ...). $$
Then, $S$ is an isometry with resolvent given by
$$ (I-\lambda S)^{-1}(x_0,x_1, \cdots) =
(x_0, x_1+\lambda x_0, \cdots, x_n+\lambda x_{n-1}+\cdots \lambda^{n-1}x_1+\lambda^nx_0,\cdots)$$
for any $x \in \ell^2_{\N}(\e)$ and any $\lambda\in\D$. We have
$$(S-T')x = (0,(I-A)x_0,0,0,\cdots), \quad \|x\|^2 - \|T'x\|^2 = \|x_0\|^2 - \|Ax_0\|^2$$
and
$$ (I-\lambda S)^{-1}(S-T')x = (0,(I-A)x_0,\lambda (I-A)x_0,\lambda^2(I-A)x_0,\cdots) .$$
Therefore the resolvent condition of Theorem~\ref{thm:res} implies that $S$ is
Harnack dominated by $T'$ if and only if $A$ is
a \emph{Halperin contraction}, that is $A$
verifies the following condition
\begin{equation} \label{eq:halperin}
\textrm{there is } K \ge 0 \textrm{ such that }
\|x_0-Ax_0\|^2 \le K(\|x_0\|^2 - \|Ax_0\|^2) \quad (x_0\in \e).
 \end{equation}
This condition was introduced by I. Halperin in \cite{H}; we refer the reader to \cite{BL} and the
references therein for more information. In particular, a product of orthogonal projections satisfies (\ref{eq:halperin}).

In our context, one sees that~\eqref{eq:halperin} is equivalent to $I\PrecZ A$. In particular,
any strict contraction $A$ satisfies it, and this yields
 another proof of the fact that a shift operator (of arbitrary multiplicity)
is not a maximal element for the Harnack relation.

We remark that any contraction which is Z-equivalent to, or Z-dominates a Halperin contraction
also verifies~\eqref{eq:halperin}.
On the other hand, it is clear that an operator $T$ with $\|T\|=1$ and $\sigma(T) \subset \D$ cannot
be a Halperin contraction since $I-T$ is invertible while $D_T$ is not. The latter statement follows from
$\|T\|=1$. But $T \Prec 0$ (see \cite{AST}, \cite{Sn}), hence a Halperin contraction
can Harnack dominates a contraction which does not necessarily satisfy~\eqref{eq:halperin}.
However, by Corollary~\ref{co47}, a contraction which Harnack dominates a Halperin contraction
certainly satisfies the Katznelson-Tzafriri
condition $\sigma(T) \subset \D \cup \{1\}$.
Indeed, this spectral condition is satisfied by any
Halperin contraction as was proved in \cite{BL}.
\end{example}

\begin{example}\label{ex:pr59 revisited}
	Suppose $\alpha\in\overline{\D}$, and let
$Z$ denotes multiplication by the variable $\zeta=e^{it}$ on the
space $\h=L^2([0,2\pi], dm)$ ($dm$ being normalized Lebesgue measure).
Define the operators $T(\alpha)=Z-(1-\alpha) Z1\otimes 1$;
by Lemma~\ref{le:one dim pert of isometries} they are contractions
for all $\alpha\in\D$ and unitary for $|\alpha|=1$. One can see
that $Z=T(1)$, while the proof of Theorem~\ref{th:abs cont unitaries are dominated},
in the case $\omega=[0,2\pi]$, shows that $Z\Prec T(0)$. To discuss in
more detail the class of all $T(\alpha)$s, we need the following well known
result concerning perturbations of unitary operators. The proof, which is a
computation, can be found, for instance, in~\cite[Proposition 1.3]{NT}.
	
	\begin{lemma}\label{le:nikolski treil}
		Suppose $U$ is unitary and $T=U-b\otimes a$ for some vectors $a,b$. Let $\lambda\in\C$ be such that $I-\lambda U$ is invertible, and denote $a_\lambda=\bar{\lambda} (I-\bar{\lambda}U^*)^{-1} a$. Then $I-\lambda T$ is invertible if and only if $1+\langle b, a_\lambda\rangle\not= 0$, in which case we have
		\begin{equation}\label{eq:rezolvent of T in terms of U}
		(I-\lambda T)^{-1} =(I-\lambda U)^{-1} \left( I-\frac{1}{1+\langle b, a_\lambda\rangle } b\otimes a_\lambda\right).
		\end{equation}
	\end{lemma}

We want to apply this result to obtain $(I-\bar{\lambda}T(\alpha))^{-1}$ for $\lambda\in\D$. We take $U=Z$,  $a(\zeta)=1$, $b(\zeta)=(1-\alpha)Za=(1-\alpha)\zeta$. In this case $a_{\bar\lambda}(\zeta)=\lambda(1-\lambda\bar{\zeta})^{-1}\in\overline{H^2}$, so $\langle b, a_\lambda\rangle=0$,
$I-\bar\lambda T(\alpha)$ is invertible and
\begin{equation}\label{eq:rezolvent for T(alpha)_lambda}
	((I-\bar\lambda T(\alpha))^{-1}f)(\zeta)=\frac{1}{1-\bar{\lambda}\zeta}\Big( f(\zeta) -\langle f,\lambda(1-\lambda\bar{\zeta})^{-1}\rangle (1-\alpha)\zeta\Big).
\end{equation}

\begin{proposition}\label{pr:T_alpha}
	All contractions $T(\alpha)$ with $|\alpha|<1$ are Harnack equivalent, and they all Harnack dominate the unitary operators  $T(\alpha')$ with $|\alpha|=1$ (in particular, they dominate $Z$).
\end{proposition}	

\begin{proof}
	Take $\alpha, \alpha'\in \overline{\D}$, and denote, to simplify notation, $T=T(\alpha), T'=T(\alpha')$.
To discuss Harnack domination, we intend to apply Theorem~\ref{thm:1.3}, so we have to make some computations related to the M\"obius transforms of $T$ and $T'$.  First, by Lemma~\ref{le:one dim pert of isometries} we have $D_{T}=(1-|\alpha|^2)1\otimes 1$ and thus
\begin{equation}\label{eq:defect T}
	\|D_{T}f\|^2=(1-|\alpha|^2) |\langle f, 1\rangle|^2.
\end{equation}
Since $\|D_{T_\lambda}f\|^2=(1-|\lambda|^2)\|D_T(I-\bar{\lambda}T)^{-1}f\|^2$, while, by~\eqref{eq:rezolvent for T(alpha)_lambda},
\[
\langle (I-\bar{\lambda}T)^{-1}f, 1\rangle =\langle (1-\bar{\lambda}\zeta)^{-1}f, 1\rangle,
\]
we have
\begin{equation}\label{eq:defect T_lambda}
\|D_{T_\lambda}f\|^2 = (1-|\lambda|^2)(1-|\alpha|^2) |\langle (1-\bar{\lambda}\zeta)^{-1}f, 1\rangle|^2.
\end{equation}
	Similarly,
\begin{equation}\label{eq:defect T'_lambda}
\|D_{T'_\lambda}f\|^2 = (1-|\lambda|^2)(1-|\alpha'|^2) |\langle (1-\bar{\lambda}\zeta)^{-1}f, 1\rangle|^2.
\end{equation}	
	
	From~\eqref{eq:T'lambda-Tlambda} and~\eqref{eq:rezolvent for T(alpha)_lambda} it follows that
\begin{equation}\label{eq:norm Tlambda-T'lambda}
\begin{split}
	\|	(T_\lambda - T'_\lambda)f\|^2& = \frac{(1-|\lambda|^2)^2}{|\lambda|^2} \| (I-\overline{\lambda} T)^{-1}f
	- (I-\overline{\lambda} T')^{-1}f \|^2\\
&	=\frac{(1-|\lambda|^2)^2}{|\lambda|^2}
\|\frac{\zeta}{1-\bar{\lambda}\zeta} \langle f,\lambda(1-\lambda\bar{\zeta})^{-1}\rangle (\alpha'-\alpha) \|^2\\
&= (1-|\lambda|^2)|\alpha'-\alpha|^2  |\langle (1-\bar{\lambda}\zeta)^{-1}f, 1\rangle|^2.
\end{split}
\end{equation}

It follows now from~\eqref{eq:defect T_lambda}, \eqref{eq:defect T'_lambda}, and \eqref{eq:norm Tlambda-T'lambda} that if $|\alpha|<1$ and $|\alpha'|\le 1$, then $T'_\lambda\PrecZ T_\lambda$ with constants independent of $\lambda$. By Theorem~\ref{thm:1.3} this proves the proposition.
\end{proof}

\end{example}

Theorem \ref{pr14} yields several properties of contractions that are preserved by Harnack domination. We will see below some other that are not necessarily preserved.

As seen in Theorem \ref{pr14},   strong stability is preserved by  Harnack domination in both senses.
This property appears in the canonical triangulation of a contraction $T$:
it is known from \cite{SFb} that $T$ has on $\h=\n(S_T)\oplus \overline{\R(S_T)}$ a triangulation of the form
\[
T=
\begin{pmatrix}
Q & \star\\
0 & W
\end{pmatrix}
\]
where $Q$ is of class $C_{0\cdot}$ on $\n(S_T)$ and $W$ is of
class $C_{1\cdot}$ on $\overline{\R(S_T)}$.

As we will show below, in contrast to  $C_{0\cdot}$, the class $C_{1\cdot}$ and the related ones $C_{\cdot1}$ and $C_{11}$ are not in general preserved by Harnack equivalence.

\begin{example}
		We will now look at Example~\ref{ex:pr59 revisited} from a different perspective.
		By considering the standard isomorphism between $L^2([0,2\pi], dm)$ and $\ell^2_{\Z}$, one may describe it in terms of weighted bilateral shifts. Moreover, since Harnack domination is preserved by taking direct sums, one can also consider vector valued sequence spaces $\ell^2_{\Z}(\e)$. We define then, for $\alpha\in\bar{\D}$, the contractions $\tau(\alpha)$ by
		$$
		\tau(\alpha)(...,h_{-1},\boxed{h_0},h_1,...)=(...,h_{-2},\boxed{h_{-1}}, \alpha h_0, h_1,...)
		$$
		for $\{h_n\}_{n\in \Z}\in \ell^2_{\Z}(\e)$. Here the components of a vector in $\ell^2_{\Z}(\e)$
		are arranged in order of increasing subscripts,
		the central component (i.e., the one with subscript $0$) being framed in a box.
		
		Then $\tau(\alpha)$ is unitarily equivalent to $T(\alpha)$. So all $\tau(\alpha)$s are Harnack equivalent for $|\alpha|<1$, and they all dominate the unitary operators $\tau(\alpha)$ with $|\alpha|=1$ (in particular the multivariate bilateral shift, which corresponds to $\alpha=1$).
		
		This approach allows us to obtain more properties of $\tau(\alpha)$. Thus, for $|\alpha|<1$, $\tau(\alpha)$ is completely nonunitary, since one sees easily that for a nonzero element $x\in\ell^2_{\Z}(\e)$  we cannot have $\|\tau(\alpha)^n x\|=\|\tau(\alpha)^*{}^n x\|=\|x\|$ for all $n\in\N$. For $\alpha\not=0$ $\tau(\alpha)$ is invertible, while $\tau(0)$ is unitarily equivalent to the partial isometry $S\oplus S^*$. In particular, this shows, in contrast to Theorem~\ref{pr812}, that the whole spectrum is not preserved by Harnack equivalence, since $0\in\sigma(\tau(0))$, but $0\not\in\sigma(\tau(\alpha))$ for $\alpha\not=0$.
		
		According to \cite{SFb}, a contraction $T$ is called a {\it weak contraction} if $\sigma(T)$
		does not fill in the closed unit disc $\overline{\D}$ and its defect operator $D_T$ is of finite trace.
		If $\dim\e<\infty$, then $\tau(\alpha)$ is a weak contraction only for $\alpha\not=0$, but not for $\alpha=0$.
So weak contractions are not preserved by Harnack equivalence.
		
		We may also compute the asymptotic limit $S_{\tau(\alpha)}$. Indeed, we have
		$$
		\tau(\alpha)^{*n}\tau(\alpha)^nh=(...,h_{-n},|\alpha |^2h_{-n+1},...,\boxed{|\alpha|^2h_0},h_1,h_2,...)
		$$
		and consequently
		$$
		S_{\tau(\alpha)}h=(...,|\alpha|^2h_{-n},...,\boxed{|\alpha|^2h_0},h_1,h_2,...)
		$$
		for $h=\{h_n\}\in \ell^2_{\Z}(\e)$, $\alpha \in \D$. The two operators displayed above are
		diagonal with respect to the standard basis of $\ell^2_{\Z}(\e)$. For $\alpha=0$ the operator $S_{\tau(\alpha)}$ is thus a
nontrivial orthogonal projection, while for $\alpha\not=0$ it is an invertible positive operator.
This is equivalent to saying that $\tau(\alpha)$ is of class $C_{1\cdot}$ for $\alpha\not=0$,
but not for $\alpha=0$. Therefore the class $C_{1\cdot}$ is not preserved by Harnack equivalence.
One can show similarly that $\tau(\alpha)$ is actually in $C_{11}$, but $\tau(0)$ is neither
in $C_{1\cdot}$ nor in $C_{\cdot1}$. Also, the class of operators whose asymptotic limit is an
orthogonal projection is not preserved by Harnack equivalence.

Denote now $T=\tau(0)$. With respect to the decomposition  $\ell^2_{\Z}(\e)= \n(I-S_T)\oplus\n(S_T)$ we may
write $T=S\oplus S^*$, where $S$ is the unilateral shift of multiplicity $\dim \e$. We can then
obtain some more information on the Harnack class of $T$.

\begin{proposition}\label{pr510}
	Each contraction $T'$ in the  Harnack part of $T$ has the form
	$$
	T'=
	\begin{pmatrix}
	S & W\\
	0 & S^*
	\end{pmatrix}
	$$
	with $S^*W=S^*W^*=0$.
\end{proposition}

\begin{proof}
	Let $T'$ be in the Harnack part of $T$. Then by Lemma \ref{pr11} one has $\n(I-S_T)=\n(I-S_{T'})$
	and $T'=T=S$ on this kernel. Also, by Lemma \ref{le:simple properties of Harnack domination} (iii),
	$T'^*|_{\n(S_T)}$ is Harnack equivalent to $S$, hence $T'^*=S$ on $\n(S_T)$. We conclude that $T'$ has the
	desired matrix representation. For $T'$ to be a contraction, one checks easily that we must have $S^*W=S^*W^*=0$.
\end{proof}

\end{example}

\begin{remark}\label{re513}

Proposition \ref{pr510} gives the matrix structure of contractions in the Harnack part of $T=\tau(0)$.
The condition $S^*W=S^*W^*=0$ means that with respect to the decomposition $\h=\n (S^*)\oplus \n (S^*)^\perp $ we have $W=W_0\oplus 0$, with $W_0$ contractive. It is necessary, but in general not sufficient for $T'$ to be Harnack equivalent to $T$. The case $T'=T(\alpha)$, with $|\alpha|<1$, corresponds to $W_0=\alpha I_{\n( S^*)}$.
If $\mathcal{E}=\C$ we obtain then that the Harnack part of $T(0)$  is precisely
the set of  $T(\alpha)$ with $|\alpha|<1$. It would be interesting to characterize
in the general case $\dim\e>1$ the class of $W_0$ for which $T'$ is in the Harnack part of $T$.
\end{remark}

\end{document}